\documentclass[11pt]{article}
\usepackage{times}
\usepackage{amsmath,amsfonts,amssymb,amsthm}
\usepackage{calrsfs}

\newtheorem{Theorem}{Theorem}
\newtheorem{Lemma}[Theorem]{Lemma}

\newtheorem{Corollary}[Theorem]{Corollary}
\renewcommand\P{{\mathbb P}}
\newcommand\E{{\mathbb E}}
\newcommand\R{{\mathbb R}}

\newcommand\Z{{\mathbb Z}}

\begin{document}
\title{Slowly varying asymptotics\\ 
for signed stochastic difference equations}
\author{Dmitry Korshunov}
\maketitle

\begin{abstract}
For a stochastic difference equation $D_n=A_nD_{n-1}+B_n$ which 
stabilises upon time we study tail distribution asymptotics of $D_n$
under the assumption that the distribution of $\log(1+|A_1|+|B_1|)$
is heavy-tailed, that is, all its positive exponential moments are infinite.
The aim of the present paper is three-fold.
Firstly, we identify the asymptotic behaviour not only of the stationary
tail distribution but also of $D_n$.
Secondly, we solve the problem in the general setting
when $A$ takes both positive and negative values.
Thirdly, we get rid of auxiliary conditions like
finiteness of higher moments used in the literature before.

\end{abstract}

\section{Introduction}

Let $(A,B)$ be a random vector in $\R^2$ such that $\E\log|A|=-a<0$.
Let $(A_k,B_k)$, $k\in\Z$, be independent copies of $(A,B)$.
Consider the following stochastic difference equation
\begin{eqnarray}\label{eq:D.D}
D_n &=& A_nD_{n-1}+B_n\nonumber\\
&=& \Pi_1^n D_0+\sum_{k=1}^n \Pi_{k+1}^n B_k,\quad n\ge 1,
\end{eqnarray}
where $D_0$ is independent of $(A_k,B_k)$'s,
$\Pi_k^n:=A_k\cdot\ldots\cdot A_n$ for $k\le n$ and $\Pi_{n+1}^n=1$.
The process $D_n$ clearly constitutes a Markov chain and satisfies
the following equality in distribution
\begin{eqnarray*}
D_n &=_{st}& \Pi_{-n}^{-1}D_0 +\sum_{k=-n}^{-1} \Pi_{k+1}^{-1} B_k.
\end{eqnarray*}
If $a<\infty$ then, by the strong law of large numbers applied 
to the logarithm of $|\Pi|$, with probability $1$, 
$e^{-2an}\le \Pi_1^n\le e^{-an/2}$ ultimately in $n$, hence the process $D_n$, 
$n\ge 1$, is stochastically bounded if and only if $\E\log(1+|B|)<\infty$.
If $\P\{A=0\}>0$ which implies $a=\infty$,
then the process $D_n$ is always stochastically bounded.
In both cases, the Markov chain $D_n$ is stable, 
its stationary distribution is given by the following random series
\begin{eqnarray*}
D_\infty\ :=\ \sum_{k=-\infty}^{-1} \Pi_{k+1}^{-1} B_k &=_{st}&
\sum_{k=1}^\infty \Pi_1^{k-1} B_k
\end{eqnarray*}
and $D_n$ weakly converges to the stationary distribution as $n\to\infty$;
in the context of financial mathematics such random variables are
called stochastic perpetuities. 
Stability results for $D_n$ are dealt with in \cite{Vervaat}, see also \cite{BDM}; 
the case where $\E\log|A|$ is not necessarily finite is treated 
in \cite{GoldieMaller}.

Both perpetuities and stochastic difference equations have many
important applications, among them life insurance and finance,
nuclear technology, sociology, 
random walks and branching processes in random environments, 
extreme-value analysis, one-dimensional ARCH processes, etc. 
For particularities, we refer the reader to, for instance, 
Embrechts and Goldie \cite{EG}, Rachev and Samorodnitsky \cite{RS} 
and Vervaat \cite{Vervaat} for a comprehensive survey of the literature.

If $A\ge 0$ and $\P\{A>1\}>0$, then $\E A^\gamma\to\infty$ as $\gamma\to\infty$,
so $\E A^\beta\ge 1$ for some $\beta<\infty$. 
If in addition $B\ge 0$ and $\P\{B>0\}>0$,
then $\E D_\infty^\beta>\E D_\infty^\beta\E A^\beta\ge \E D_\infty^\beta$
which implies that $\E D_\infty^\beta=\infty$, in other words,
with necessity, not all moments of $D_\infty$ are finite;
see \cite{GG1996} for a similar conclusion for signed $A$ and $B$.
It was proven in the seminal paper by Kesten \cite[Theorem 5]{Kesten1973}, 
see also \cite{Goldie1991},
that if $\E |A|^\beta=1$ for some $\beta>0$, then
a power tail asymptotics for the stationary distribution holds,
$\P\{|D_\infty|>x\}\sim c/x^\beta$ as $x\to\infty$, for some $c>0$.

The problem we address in this paper is about the tail asymptotic behaviour 
of $D_n$ and of its stationary version $D_\infty$ in the case where 
the distribution of $\log|A|$ is heavy-tailed, that is,
all positive exponential moments of $\log|A|$ are infinite,
in other words, $\E|A|^\gamma=\infty$ for all $\gamma>0$.
It can only happen if the random variable $|A|$ has right unbounded support.

The only result in that direction we are aware of is that by Dyszewski 
\cite{Dyszewski} where in the context of iterated random functions 
it is proven that the stationary tail distribution is
asymptotically equivalent to
$$
\frac{1}{a}\int_x^\infty \P\{\log C>y\}dy\quad\mbox{as }x\to\infty,
$$
where $C:=\max(A,B)$, provided $A$, $B\ge 0$, the integrated tail distribution 
of $\log C$ is subexponential and under additional moment condition that
$\E\log^{1+\gamma} C<\infty$ for some $\gamma>0$. In the case of a signed $B$, 
only lower and upper asymptotic bounds are derived in \cite{Dyszewski}.
An alternative approach to lower and upper bounds for the tail of $D_\infty$
is developed in \cite{CRZ} in the case of positive $A$ and $B$.

The aim of the present paper is three-fold.
Firstly, we identify the asymptotic behaviour not only of the stationary
tail distribution but also of $D_n$ in the heavy-tailed case.
Secondly, we solve the problem in the general setting
when $A$ takes both positive and negative values.
Thirdly, we get rid of auxiliary conditions like
finiteness of higher moments.

Our approach to the problem is based on reduction of $D_n$ -- roughly
speaking by taking the logarithm of it -- to
an asymptotically homogeneous in space Markov chain
with heavy-tailed jumps and on further analysis of such chains.
Namely, we define a Markov chain $X_n$ on $\R$ as follows
\begin{eqnarray}\label{def.X.signed}
X_n &:=& \left\{
\begin{array}{rl}
\log(1+D_n) &\mbox{ if }D_n\ge 0,\\
-\log(1+|D_n|) &\mbox{ if }D_n< 0,
\end{array}
\right.
\end{eqnarray}
hence the distribution tail of $D_n$ may be computed as
\begin{eqnarray}\label{D.via.X}
\P\{D_n>x\} &=& \P\{X_n>\log(1+x)\}\quad\mbox{for }x>0.
\end{eqnarray}
At any state $x\ge 0$, the jump of the Markov chain $X_n$ is 
a random variable distributed as
\begin{eqnarray}\label{def.xi.signed+}
\xi(x) &=& 
\left\{
\begin{array}{rl}
\log(1+A(e^x{-}1)+B)-x & \mbox{ if }A(e^x{-}1)+B\ge 0,\\
-\log(1+|A(e^x{-}1)+B|)-x & \mbox{ if }A(e^x{-}1)+B<0,
\end{array}
\right.
\end{eqnarray}
and at any state $x\le 0$,
\begin{eqnarray}\label{def.xi.signed-}
\xi(x) &=& 
\left\{
\begin{array}{rl}
\log(1+A(1{-}e^{-x})+B)-x & \mbox{ if }A(1{-}e^{-x})+B\ge 0,\\
-\log(1+|A(1{-}e^{-x})+B|)-x & \mbox{ if }A(1{-}e^{-x})+B<0.
\end{array}
\right.
\end{eqnarray}
Also define a sequence of independent random fields $\xi_n(x)$, $x\in\R$,
which are independent copies of $\xi(x)$.
Then the recursion \eqref{eq:D.D} may be rewritten as
\begin{eqnarray*}
X_{n+1} &=& X_n+\xi_n(X_n).
\end{eqnarray*}
The Markov chain $X_n$ is {\it asymptotically homogeneous in space}, 
that is, the distribution of its jump $\xi(x)$ 
weakly converges to that of $\xi:=\log A$ as $x\to\infty$; 
it is particularly emphasised in \cite[Section 2]{Goldie1991}.
Let us underline that, in general, $\log(A+(1-A+B)e^{-x})$ 
may not converge to $\xi$ as $x\to\infty$ in total variation norm.

Asymptotically homogeneous in space Markov chains are studied in detail 
in \cite{BK2002,K2002} from the point of view of their asymptotic tail behaviour
in subexponential case.
However, that results for general asymptotically homogeneous in space Markov 
chains are not directly applicable to stochastic difference equations
as it is formally assumed in \cite[Theorem 3]{BK2002}
that the distribution of a Markov chain $X_n$ converges to
the invariant distribution in total variation norm
which is not always true for stochastic difference equations.
Secondly, stochastic difference equations possess some specific properties 
that allow us to find tail asymptotics in a simpler way 
than it is done in \cite[Theorem 3]{BK2002}
or in \cite[Theorem 3.1]{Dyszewski}; 
we explore that below however our approach still follows some ideas 
of the proof for Markov chains in \cite{BK2002}.

Let us recall some relevant classes of distributions needed 
in the heavy-tailed case.

{\bf Definition 1.} A distribution $H$ with right unbounded support 
is called {\it long-tailed}, $H\in\mathcal L$, if, for each fixed $y$,
$\overline H(x+y)\sim\overline H(x)$ as $x\to\infty$;
hereinafter $\overline H(x)=H(x,\infty)$ is the tail of $H$.

A random variable $A>0$ has slowly varying at infinity distribution
if and only if the distribution of $\xi:=\log A$ is long-tailed.

{\bf Definition 2.} A distribution $H$ on $\R^+$
with unbounded support is called {\it subexponential}, $H\in\mathcal S$,
if $\overline{H*H}(x)\sim 2\overline H(x)$ as $x\to\infty$.
Equivalently,
$\P\{\zeta_1+\zeta_2>x\}\sim 2\P\{\zeta_1>x\}$,
where random variables $\zeta_1$ and $\zeta_2$
are independent with common distribution $H$.
A distribution $H$ of a random variable $\zeta$ on $\R$ 
with right-unbounded support is called {\it subexponential} 
if the distribution of $\zeta^+$ is so.

As well-known (see, e.g. \cite[Lemma 3.2]{FKZ}) the
subexponentiality of $H$ on $\R^+$ implies long-tailedness of $H$.
In~particular, if the distribution of a random variable 
$\zeta\ge 0$ is subexponential then $\zeta$ is heavy-tailed.

For a distribution $H$ with finite mean, we define the
{\it integrated tail distribution} $H_I$ generated by $H$ as follows:
$$
\overline H_I(x) :=
\min\Bigl(1,\int_x^\infty \overline H(y)dy\Bigr).
$$

{\bf Definition 3.} A distribution $H$ on $\R^+$
with unbounded support and finite mean is called {\it strong subexponential},
$H\in\mathcal S^*$, if 
\begin{eqnarray*}
\int_0^x\overline H(x-y)\overline H(y)dy &\sim& 2m\overline H(x)
\quad\mbox{as }x\to\infty,
\end{eqnarray*}
where $m$ is the mean value of $H$. It is known that if $H\in\mathcal S^*$ 
then both $H$ and $H_I$ are subexponential distributions,
see e.g. \cite[Theorem 3.27]{FKZ}.

In what follows we use the following notation for distributions: we denote  

(i) the distribution of $\log(1+|A|+|B|)$ by $H$;

(ii) the distribution of $\log(1+|A|)$ by $F$;

(iii) the distribution of $\log(1+|B|)$ by $G$;

(iv) the distribution of $\log(1+B^+)$ by $G^+$;

(v) the distribution of $\log(1+B^-)$ by $G^-$.

The paper is organised as follows. 
In Sections \ref{sec:subexp.1}, \ref{sec:subexp.2} and \ref{sec:subexp.3}
we assume that $\log |A|$ has finite negative mean and
successively investigate three different cases in the order of increasing difficulty:
(i) both $A$ and $B$ are positive, see Theorem \ref{thm:subexp}; 
(ii) $A$ is positive and $B$ is a signed random variable, 
see Theorem \ref{thm:subexp.gen.B};
(iii) both $A$ and $B$ are signed, see Theorem \ref{thm:subexp.gen}.
In the case (i) we also explain in Theorem \ref{th:psbj}
the most probable way by which large deviations of $D_n$ can occur
-- it is a version of the principle of a single big jump
playing the key role in the theory of subexponential distributions.
The aim of Section \ref{sec:subexp.1.5} is to explain what happens 
if the distribution of $A$ has an atom at zero;
in that case the tail asymptotics of $D_n$ is essentially different 
from what we observe if $A$ has no atom at zero.

\section{Positive stochastic difference equation}
\label{sec:subexp.1}

In this section we consider a positive $D_n$, 
so $A>0$, $B\ge 0$ -- we exclude the case where $A$ has
an atom at zero as then the tail asymptotics of $D_n$ 
are essentially different, see the next section. 
Then the Markov chain $X_n:=\log(1+D_n)$ is positive too.
As above, we denote $\xi:=\log A$ and the distribution of the
random variable $\log(1+A+B)$ by $H$.

\begin{Theorem}\label{thm:subexp}
Suppose that $A>0$, $B\ge 0$, $\E\xi=-a\in(-\infty,0)$ and $\E\log(1+B)<\infty$,
so that $D_n$ is positive recurrent.

If the integrated tail distribution $H_I$ is long-tailed, then
\begin{eqnarray}\label{thm.lower.1}
\P\{D_\infty>x\} &\ge& (a^{-1}+o(1))\overline{H_I}(\log x)\quad\mbox{as }x\to\infty.
\end{eqnarray}
If, in addition, the distribution $H$ is long-tailed itself, then
\begin{eqnarray}\label{thm.lower.2}
\P\{D_n>x\} &\ge& \frac{1{+}o(1)}{a}\int_{\log x}^{\log x+na}\overline H(y)dy
\ \mbox{as }x\to\infty\mbox{ uniformly for all }n\ge 1.\nonumber\\[-2mm]
\end{eqnarray}

If the integrated tail distribution $H_I$ is subexponential then
\begin{eqnarray}\label{thm.upper.1}
\P\{D_\infty>x\} &\sim& a^{-1}\overline H_I(\log x)
\quad\mbox{as }x\to\infty.
\end{eqnarray}
If moreover the distribution $H$ is strong subexponential then
\begin{eqnarray}\label{thm.upper.2}
\P\{D_n>x\} &\sim& \frac{1}{a}\int_{\log x}^{\log x+na}\overline H(y)dy
\quad\mbox{as }x\to\infty\mbox{ uniformly for all }n\ge 1.\nonumber\\[-2mm]
\end{eqnarray}
\end{Theorem}

The main contribution of Theorem \ref{thm:subexp} is \eqref{thm.upper.2}
that states uniform asymptotic behaviour for all $n\ge 1$. 
It is much stronger than a rather simple conclusion that 
\eqref{thm.upper.2} holds for a fixed $n$ 
earlier proven by Dyszewski in \cite[Theorem 3.3]{Dyszewski}
by induction argument that clearly does not work
if we wanted to describe tail asymptotics for the entire range of $n\ge 1$.

In \cite{Dyszewski}, a sufficient condition for the asymptotics
\eqref{thm.upper.1} is formulated in terms of the distribution of $\log\max(A,B)$ 
instead of $H$. Let us show that these two approaches are equivalent.
Indeed, for any two positive random variables $A$ and $B$, since
\begin{eqnarray*}
\max(\log(1+A),\log(1+B)) &\le& \log(1+A+B)\\ 
&<& \log 2+\max(\log(1+A),\log(1+B)),
\end{eqnarray*}
it follows that 

(i) the distribution $H$ is long-tailed/subexponential/strong subexponential
if and only if the distribution of $\max(\log(1+A),\log(1+B))$ 
is long-tailed/sub\-exponential/strong subexponential respectively;

(ii) the distribution $H_I$ is subexponential if and only if the integrated tail 
distribution of $\max(\log(1+A),\log(1+B))$ is so.

Denote the distribution of $\log(1+A)$ by $F$ and that of $\log(1+B)$ by $G$.
In the next result we discuss some sufficient conditions for 
subexponentiality and related properties of $H$.

\begin{Lemma}
Let $A$ and $B$ be any two positive random variables such that 
either of the following two conditions holds:

(i) the distribution $H$ of $\log(1+A+B)$ is long-tailed or

(ii) the random variables $A$ and $B$ are independent.

Then if the distribution $(F+G)/2$ is subexponential or strong subexponential, 
then the distribution $H$ is subexponential or strong subexponential respectively.

If the integrated tail distribution $(F_I+G_I)/2$ is subexponential, 
then $H_I$ is subexponential too.
\end{Lemma}

\begin{proof}
First assume that (i) holds. On one side,
\begin{eqnarray}\label{1AB.lower}
\overline H(x) &=& \P\{\log(1+A+B)>x\}\nonumber\\ 
&\ge& \frac{\P\{\log(1+A)>x\}+\P\{\log(1+B)>x\}}{2}\nonumber\\
&=& \bigl(\overline F(x)+\overline G(x)\bigr)/2
\end{eqnarray}
and thus, for all sufficiently large $x$,
\begin{eqnarray}\label{1AB.I.lower}
\overline{H_I}(x) &\ge& \bigl(\overline{F_I}(x)+\overline{G_I}(x)\bigr)/2.
\end{eqnarray}
On the other side,
\begin{eqnarray}\label{H.above.gen}
\overline H(x) &\le& \P\{\log(1+2A)>x\}+\P\{\log(1+2B)>x\}\nonumber\\
&<& \overline F(x-\log 2)+\overline G(x-\log 2).
\end{eqnarray}
If $(F+G)/2$ is subexponential then it is long-tailed and hence
\begin{eqnarray}\label{1AB.upper}
\overline H(x) &\le& 
(1+o(1))\bigl(\overline F(x)+\overline G(x)\bigr)\quad\mbox{as }x\to\infty.
\end{eqnarray}
If $(F_I+G_I)/2$ is subexponential then similarly
\begin{eqnarray}\label{1AB.I.upper}
\overline{H_I}(x) &\le& 
(1+o(1))\bigl(\overline{F_I}(x)+\overline{G_I}(x)\bigr)\quad\mbox{as }x\to\infty.
\end{eqnarray}
The two bounds \eqref{1AB.upper} and \eqref{1AB.lower} 
in the case of long-tailed $H$ allow us to apply Theorem 3.11 or 3.25 
from \cite{FKZ} and to conclude subexponentiality
or strong subexponentiality of $H$ respectively provided $(F+G)/2$ is so.

The two bounds \eqref{1AB.I.upper} and \eqref{1AB.I.lower} 
in the case of long-tailed $H_I$ allow us to apply 
Theorem 3.11 from \cite{FKZ} and to conclude
subexponentiality of $H_I$ provided $(F_I+G_I)/2$ is so.

Now let us consider the case where $A$ and $B$ are independent
which yields the following improvement on the lower bound 
\eqref{1AB.lower}. For all $x>0$, 
\begin{eqnarray*}
\overline H(x) &\ge& \P\{\log(1+A)>x\}+\P\{\log(1+A)\le x\}\P\{\log(1+B)>x\}\nonumber\\
&=& \overline F(x)+F(x)\overline G(x)\\
&\sim& \overline F(x)+\overline G(x)\quad\mbox{as }x\to\infty.
\end{eqnarray*}
Therefore, $H$ inherits the tail properties of the distribution $(F+G)/2$,
and $H_I$ the tail properties of $(F_I+G_I)/2$.
\end{proof}

\begin{proof}[Proof of Theorem \ref{thm:subexp}]
At any state $x\ge 0$, the Markov chain $X_n$ has jump
\begin{eqnarray*}
\xi(x) &=& \log(1+A(e^x-1)+B)-x\\
&=& \log(A+e^{-x}(1-A+B))\\
&\ge& \log(A-e^{-x}A),
\end{eqnarray*}
as $B\ge 0$. Fix an $\varepsilon>0$. Choose $x_0$ sufficiently large 
such that $\log(1-e^{-x_0})\ge -\varepsilon/2$.
Then the family of jumps $\xi(x)$, $x\ge x_0$, possesses an integrable minorant
\begin{eqnarray}\label{xi.x.min}
\xi(x) &\ge& \xi+\log(1-e^{-x_0})\nonumber\\ 
&\ge& \xi-\varepsilon/2\ =:\ \eta.
\end{eqnarray}
On the other hand, since $A>0$ and $B\ge 0$, 
the family of jumps $\xi(x)$, $x\ge x_0$, 
possesses an integrable majorant $\zeta(x_0):=\log(A+e^{-x_0}(1+B))$.
For a sufficiently large $x_0$,
\begin{eqnarray}\label{xi.x.maj}
\E\log(A+e^{-x_0}(1+B)) &\le& \E\xi+\varepsilon,
\end{eqnarray}
owing to the dominated convergence theorem which applies because
firstly $\log(A+e^{-x_0}(1+B))\to \log A=\xi$ a.s. as $x_0\to\infty$
and secondly, by the concavity of the function $\log(1+z)$,
\begin{eqnarray*}
\log(A+e^{-x_0}(1+B)) &<& \log(1+A+e^{-x_0}(1+B))\\
&\le& \log(1+A)+\log(1+e^{-x_0}(1+B)),
\end{eqnarray*}
which is integrable by the finiteness of $\E\xi$ and $\E\log(1+B)$.

Let us first prove the lower bound \eqref{thm.lower.1}
following the single big jump technique known from the theory of
subexponential distributions. Since $D_n$ is assumed to be convergent,
the associated Markov chain $X_n$ is stable, so there exists a $c>2$ such that
$$
\P\{X_n\in(1/c,c]\}\ \ge\ 1-\varepsilon \quad\mbox{for all }n\ge 0.
$$
Let us consider the event
\begin{eqnarray}\label{def.Bknc}
\Omega(k,n,c) &:=& \{\eta_{k+1}+\ldots+\eta_{k+j}
\ge -c-n(a+\varepsilon)\mbox{ for all }j\le n\},
\end{eqnarray}
where $\eta_k$ are independent copies of $\eta$ defined in \eqref{xi.x.min}. 
By the strong law of large numbers,
there exists a sufficiently large $c$ such that
\begin{eqnarray}\label{B.c}
\P\{\Omega(k,n,c)\} &\ge& 1-\varepsilon\quad\mbox{for all }k\mbox{ and }n.
\end{eqnarray}
It follows from \eqref{xi.x.min} that any of the events
\begin{eqnarray}\label{events.lower.bound}
\{X_{k-1}\le c,\ 
X_k>x+c+(n-k)(a+\varepsilon),\ \Omega(k,n-k,c)\}
\end{eqnarray}
implies $X_n>x$ and they are pairwise disjoint.
Therefore, by the Markov property and \eqref{B.c}, 
\begin{eqnarray*}
\lefteqn{\P\{X_n>x\}}\\ 
&\ge& \sum_{k=1}^n \P\{X_{k-1}\le c,\ 
X_k>x+c+(n-k)(a+\varepsilon)\}\P\{\Omega(k,n-k,c)\}\\
&\ge& (1-\varepsilon)\sum_{k=1}^n
\P\{X_{k-1}\in(1/c,c],\ X_k>x+c+(n-k)(a+\varepsilon)\}.
\end{eqnarray*}
The $k$th probability on the right hand side equals
\begin{eqnarray*}
\lefteqn{\int_{1/c}^c\P\{X_{k-1}\in dy\}\P\{y+\xi(y)>x+c+(n-k)(a+\varepsilon)\}}\\
&=&  \int_{1/c}^c
\P\{X_{k-1}\in dy\}\P\{\log(1+A(e^y-1)+B)>x+c+(n-k)(a+\varepsilon)\}.
\end{eqnarray*}
For all $y>1/c$, 
\begin{eqnarray*}
\log(1+A(e^y-1)+B) &\ge& \log(1+A(e^{1/c}-1)+B)\\
&\ge& \log(1+A+B)+\log(e^{1/c}-1),
\end{eqnarray*}
because $e^{1/c}-1<\sqrt e-1<1$. Therefore, 
the value of the last integral is not less than
\begin{eqnarray*}
\P\{X_{k-1}\in(1/c,c]\} \P\{\log(1+A+B)>x+c_1+(n-k)(a+\varepsilon)\},
\end{eqnarray*}
where $c_1:=c-\log(e^{1/c}-1)$. Hence, due to the choice of $c$,
\begin{eqnarray*}
\P\{X_n>x\} &\ge& (1-\varepsilon)^2
\sum_{k=1}^n \overline H(x+c_1+(n-k)(a+\varepsilon)).
\end{eqnarray*}
Since the tail is a decreasing function, the last sum is not less than
\begin{eqnarray}\label{integral.0.n}
\frac{1}{a+\varepsilon}
\int_0^{n(a+\varepsilon)} \overline H(x+c_1+y)dy.
\end{eqnarray}
Letting $n\to\infty$ we obtain that the tail at point $x$ 
of the stationary distribution of the Markov chain $X$ is not less than
\begin{eqnarray*}
\frac{(1-\varepsilon)^2}{a+\varepsilon}
\int_0^\infty \overline H(x+c_1+y)dy
&=& \frac{(1-\varepsilon)^2}{a+\varepsilon} \overline{H_I}(x+c_1)\\
&\sim& \frac{(1-\varepsilon)^2}{a+\varepsilon} \overline{H_I}(x)
\quad\mbox{as }x\to\infty,
\end{eqnarray*}
due to the long-tailedness of the integrated tail distribution $H_I$.
Summarising altogether we deduce that, for every fixed $\varepsilon>0$,
$$
\liminf_{x\to\infty}
\frac{\P\{D_\infty>x\}}{\overline{H_I}(\log x)}
\ge \frac{(1-\varepsilon)^2}{a+\varepsilon},
$$
which implies the lower bound \eqref{thm.lower.1} 
due to the arbitrary choice of $\varepsilon>0$.

If the distribution $H$ is long-tailed itself, then the integral in 
\eqref{integral.0.n} is asymptotically equivalent to the integral
$$
\int_x^{x+n(a+\varepsilon)} \overline H(y)dy
\quad\mbox{as }x\to\infty\mbox{ uniformly for all }n\ge 1,
$$
which implies the second lower bound \eqref{thm.lower.2}.

Now let us turn to the asymptotic upper bound under the assumption that
the integrated tail distribution $H_I$ is subexponential.
Fix an $\varepsilon\in(0,a)$. Let $x_0$ be defined as in \eqref{xi.x.maj},
so $\E\zeta(x_0)\le -a+\varepsilon$. 
Let $J$ be the distribution of $\zeta(x_0)$. Since
$$
\log(1+A+B)-x_0\ \le\ \zeta(x_0)\ \le\ \log(1+A+B),
$$
we have $\overline H(x+x_0)\le\overline J(x)\le\overline H(x)$.
Then subexponentiality of $H_I$ yields subexponentiality 
of the integrated tail distribution $J_I$ and 
$\overline J_I(x)\sim\overline H_I(x)$ as $x\to\infty$.

By the construction of $\zeta(x_0)$,
\begin{eqnarray}\label{xi.eta.1}
x+\xi(x) &\le& y+\zeta(x_0)\quad\mbox{for all }y\ge x\ge x_0.
\end{eqnarray}
Also, by the positivity of $A$,
\begin{eqnarray}\label{xi.eta.2}
x+\xi(x) &=& \log(1+A(e^x-1)+B)\nonumber\\
&\le& \log(1+A(e^{x_0}-1)+B)\nonumber\\
&=& x_0+\xi(x_0)\ \le\ x_0+\zeta(x_0)\quad\mbox{for all }x\le x_0.
\end{eqnarray}
Consider a random walk $Z_n$ delayed at the origin with jumps $\zeta(x_0)$:
$$
Z_0:=0,\ \ Z_n:=(Z_{n-1}+\zeta_n(x_0))^+,
$$
where $\zeta_n(x_0)$ are independent copies of $\zeta(x_0)$.
The upper bounds \eqref{xi.eta.1} and \eqref{xi.eta.2} yield that
the two chains $X_n$ and $Z_n$ can be constructed on a common probability space
in such a way that, with probability $1$,
\begin{eqnarray}\label{Xn.Zn.dom}
X_n &\le& x_0+Z_n\quad\mbox{for all }n,
\end{eqnarray}
so $X_n$ is dominated by a random walk on $[x_0,\infty)$ delayed at point $x_0$.
Since the integrated tail distribution $J_I$ is subexponential, 
the tail of the invariant measure of the chain $Z_n$ is asymptotically equivalent 
to $\overline J_I(x)/(a-\varepsilon)\sim\overline H_I(x)/(a-\varepsilon)$ 
as $x\to\infty$, see, for example, \cite[Theorem 5.2]{FKZ}. 
Thus, the tail of the invariant measure of $X_n$ is asymptotically 
not greater than $\overline H_I(x-x_0)/(a-\varepsilon)$
which is equivalent to $\overline H_I(x)/(a-\varepsilon)$,
since $H_I$ is long-tailed by subexponentiality. Hence,
$$
\limsup_{x\to\infty} \frac{\P\{D_\infty>x\}}{\overline H_I(\log x)}
\le \frac{1}{a-\varepsilon}.
$$
Due to arbitrary choice of $\varepsilon>0$
and the lower bound proven above this completes the proof
of the first asymptotics \eqref{thm.upper.1}.

The same arguments with the same majorant \eqref{Xn.Zn.dom} allow us to conclude 
the finite time horizon asymptotics for $D_\infty$
if we apply Theorem 5.3 from \cite{FKZ} instead of Theorem 5.2.
\end{proof}

Theorem \ref{thm:subexp} makes it possible to identify a moment of time 
after which the tail distribution of $D_n$ is equivalent to that of $D_\infty$,
in some particular strong subexponential cases.

\begin{Corollary}\label{cor:subexp.equiv}
Suppose that $\E\log A=-a<0$, $B>0$ and $\E\log(1+B)<\infty$.

If the distribution $H$ of $\log(1+A+B)$
is regularly varying at infinity with index $\alpha<-1$, 
then $\P\{D_n>x\} \sim \P\{D_\infty>x\}$ as $n$, $x\to\infty$
if and only if $n/\log x\to\infty$.

If $\overline H(x)\sim e^{-x^\beta}$ for some $\beta\in(0,1)$, 
then $\P\{D_n>x\} \sim \P\{D_\infty>x\}$ as $n$, $x\to\infty$
if and only if $n/\log^{1-\beta} x\to\infty$.
\end{Corollary}

We conclude this section by a version of the principle of a single big
jump for $D_n$. For any $c>1$ and $\varepsilon>0$ consider events
\begin{eqnarray*}
\Omega_k &:=& \{1/c<X_{k-1}\le c,\ X_k>\log x+c+(n-k)(a+\varepsilon),\\
&&\hspace{35mm} |X_{k+j}-X_k+aj|\le c+j\varepsilon\mbox{ for all }j\le n-k\bigr\}
\end{eqnarray*}
or, in terms of $D_n$,
\begin{eqnarray*}
\Omega^D_k &:=& \{1/c<D_{k-1}\le c,\ 
A_k/c+B_k>xe^{c+(n-k)(a+\varepsilon)},\ \\
&&\hspace{20mm} e^{-c-j(a+\varepsilon)}\le D_{k+j}/D_k\le e^{c-j(a-\varepsilon)}
\mbox{ for all }j\le n-k\bigr\}.
\end{eqnarray*}
Roughly speaking, it describes a trajectory such that, for large $x$,
the $D_{k-1}$ is neither too far away from zero nor too close,
then a single big jump occurs, both $A_k$ and $B_k$ may contribute to that
big jump, and then the logarithm of $D_{k+j}$, $j\le n-k$, 
moves down according to the strong law of large numbers with drift $-a$.
As stated in the next theorem, the union of all these events
describes more precisely than the lower bound of Theorem \ref{thm:subexp} 
the most probable way by which large deviations of $D_n$ do occur.

\begin{Theorem}\label{th:psbj}
Let the distribution $H$ of $\log(1+A+B)$ be strong subexponential.
Then, for any fixed $\varepsilon>0$,
\begin{eqnarray*}
\lim_{c\to\infty}\lim_{x\to\infty}
\inf_{n\ge 1}\P\{\cup_{k=0}^{n-1} \Omega_k\mid D_n>x\} &=& 1.
\end{eqnarray*}
\end{Theorem}

\begin{proof}
The events $\Omega(k)$, $k\le n$, are pairwise disjoint and any of them
implies $\{X_n>\log x\}$. Then similar arguments as in the proof of 
lower bound in Theorem \ref{thm:subexp} apply.
\end{proof}

\section{Impact of atom at zero}
\label{sec:subexp.1.5}

In this section we demonstrate what happens if the distribution of $A$ 
has an atom at zero. It turns out that then the tail asymptotics 
of $D_n$ are essentially different -- they are proportional to the tail of 
$H$ which is lighter than given by integrated tail distribution $H_I$ 
in the case where $A>0$ -- because the chain satisfies 
Doeblin's condition, see e.g. \cite[Ch. 16]{Tweedie}. 
As above, we denote by $H$ the distribution of the random variable $\log(1+A+B)$.
For simplicity, we assume that $B>0$.

\begin{Theorem}\label{thm:subexp.5}
Suppose that $A\ge 0$, $B>0$ and $p_0:=\P\{A=0\}\in(0,1)$.
If the distribution $H$ is long-tailed and $D_0>0$, then
\begin{eqnarray}\label{thm.lower.2.5}
\P\{D_n>x\} &\ge& \Bigl(\frac{1-(1-p_0)^n}{p_0}+o(1)\Bigr)\overline H(\log x)
\end{eqnarray}
as $x\to\infty$ uniformly for all $n\ge 1$. In particular, 
\begin{eqnarray}\label{thm.lower.1.5}
\P\{D_\infty>x\} &\ge& (p_0^{-1}+o(1))\overline H(\log x)\quad\mbox{as }x\to\infty.
\end{eqnarray}

If the distribution $H$ is subexponential, $D_0>0$ and
$\{D_0>x\}=o(\overline H(x))$ then
\begin{eqnarray}\label{thm.upper.2.5}
\P\{D_n>x\} &\sim& \frac{1-(1-p_0)^n}{p_0}\overline H(\log x)
\end{eqnarray}
as $x\to\infty$ uniformly for all $n\ge 1$. In particular,
\begin{eqnarray}\label{thm.upper.1.5}
\P\{D_\infty>x\} &\sim& p_0^{-1}\overline H(\log x)
\quad\mbox{as }x\to\infty.
\end{eqnarray}
\end{Theorem}

\begin{proof}
Let $H_0$ be the distribution of $\log(1+A+B)$ conditioned on $A>0$
and $G_0$ be the distribution of $\log(1+B)$ conditioned on $A=0$,
then $H=p_0G_0+(1-p_0)H_0$.

Let us decompose the event $X_n>x$ according to the last zero value of 
$A_k$, which gives equality
\begin{eqnarray}\label{Xn.last.zero}
\P\{X_n>x\} &=& \P\{A_1,\ldots,A_n>0, X_n>x\}\nonumber\\
&&\hspace{10mm}
+\sum_{k=1}^n \P\{A_k=0,A_{k+1}>0,\ldots,A_n>0, X_n>x\}\nonumber\\
&=& (1-p_0)^n\P\{X_n>x\mid A_1,\ldots,A_n>0\}\nonumber\\
&& +p_0\sum_{k=1}^n (1-p_0)^{n-k}\P\{X_n>x\mid 
A_k=0,A_{k+1},\ldots,A_n>0\}\nonumber\\
&=& (1-p_0)^n\P\{X_n>x\mid A_1,\ldots,A_n>0\}\nonumber\\
&& +p_0\sum_{k=0}^{n-1} (1-p_0)^k\P\{X_{k+1}>x\mid 
A_1=0,A_2,\ldots,A_{k+1}>0\},\nonumber\\[-3mm]
\end{eqnarray}
by the Markov property. In particular, the sum from $0$ to $n-1$
on the right hand side is increasing as $n$ grows as all terms are positive. 
For that reason, for the lower bounds for $\P\{D_n>x\}$ 
it suffices to prove by induction that, for any fixed $k\ge 0$
and $\gamma>0$, there exists a $c<\infty$ such that 
\begin{eqnarray}\label{XAAA}
\lefteqn{\P\{X_{k+1}>x\mid A_1=0,A_2,\ldots,A_{k+1}>0\}}\nonumber\\
&&\hspace{30mm}\ge\ 
(1-\gamma)\bigl(\overline{G_0}(x+c)+k\overline{H_0}(x+c)\bigr),
\end{eqnarray}
\begin{eqnarray}\label{XAAA.0}
\P\{X_{k+1}>x\mid A_1,\ldots,A_{k+1}>0\} &\ge& 
(1-\gamma)(k+1)\overline{H_0}(x+c)
\end{eqnarray}
for all sufficiently large $x$, because then
\begin{eqnarray*}
\P\{X_n>x\} &\ge& (1-\gamma)\biggl((1-p_0)^n n \overline{H_0}(x+c)\\
&&\hspace{20mm} +p_0\sum_{k=0}^{n-1} (1-p_0)^k
\bigl(\overline{G_0}(x+c)+k\overline{H_0}(x+c)\bigr)\biggr)\\
&=& (1-\gamma)(\bigl(1-(1-p_0)^n\bigr)
\Bigl(\overline{G_0}(x+c)+\frac{1-p_0}{p_0}\overline{H_0}(x+c)\Bigr)\\
&=& (1-\gamma)\frac{1-(1-p_0)^n}{p_0} \overline H(x+c),
\end{eqnarray*}
with further application of long-tailedness of $H$.

To prove \eqref{XAAA}, first let us note that the induction basis $k=0$ 
is immediate, since the distribution of $X_1$ conditioned on $A_1=0$ is $G_0$. 
Now let us assume that \eqref{XAAA} is true for some $k$. Denote 
$$
G_k(dy)\ :=\ \P\{X_{k+1}\in dy\mid A_1=0,A_2,\ldots,A_{k+1}>0\},\quad k\ge 0,
$$
which is a distribution on $(0,\infty)$. Then
\begin{eqnarray*}
\overline{G_{k+1}}(x) &=& \int_0^\infty \P\{\log(1+A(e^y-1)+B)>x\mid A>0\} G_k(dy)\\
&\ge& \int_\varepsilon^{1/\varepsilon} \P\{\log(1+A\delta+B)>x\mid A>0\} G_k(dy)\\
&&\hspace{10mm}+\int_{x+1/\varepsilon}^\infty \P\{\log(A(e^y-1))>x\mid A>0\} G_k(dy)\\
&=:& I_1+I_2,
\end{eqnarray*}
for any $\varepsilon\in(0,1/2]$ where $\delta=e^\varepsilon-1<\sqrt e-1<1$. 
Let us observe that then
\begin{eqnarray*}
\P\{\log(1+A\delta+B)>x\mid A>0\} &=& 
\P\{\log(1/\delta+A+B/\delta)>x-\log\delta\mid A>0\}\\
&\ge& \overline H_0(x-\log\delta).
\end{eqnarray*}
Therefore,
\begin{eqnarray*}
I_1 &\ge& \overline H_0(x-\log\delta) G_k(\varepsilon,1/\varepsilon].
\end{eqnarray*}
The second integral may be bounded below as follows:
\begin{eqnarray*}
I_2 &\ge& \P\{\log(A(e^{x+1/\varepsilon}-1))>x\mid A>0\} 
\overline G_k(x+1/\varepsilon)\\
&\ge& \P\{\log(A e^{x+1/2\varepsilon})>x\mid A>0\} \overline G_k(x+1/\varepsilon)\\
&=& \P\{A >e^{-1/2\varepsilon}\mid A>0\} \overline G_k(x+1/\varepsilon),
\end{eqnarray*}
for all sufficiently large $x$. Letting $\varepsilon\to 0$ we obtain that, 
for any fixed $\gamma>0$, there exists a $c<\infty$ 
such that the following lower bound holds
\begin{eqnarray*}
\overline{G_{k+1}}(x) &\ge& 
(1-\gamma)\bigl(\overline{H_0}(x+c)+\overline{G_k}(x+c)\bigr) 
\end{eqnarray*}
for all sufficiently large $x$, which implies the induction step.

The second lower bound, \eqref{XAAA.0}, follows by similar arguments
provided $D_0>0$.

Let us now proceed with a matching upper bound under the assumption
that $H$ is a subexponential distribution.
Since $A$, $B\ge 0$, 
\begin{eqnarray}\label{bound.xi.x}
\xi(x) &=& \log(A+e^{-x}(1-A+B))\\ 
&\le& \log(1+A+B)
\quad\mbox{for all }x>0.
\end{eqnarray}
Let $\eta$ and $\zeta$ be random variables with the following tail distributions
\begin{eqnarray*}
\P\{\eta>x\} &=& \min\biggl(1,\frac{\P\{\log(1+A+B)>x\}}{\P\{A=0\}}\biggr),\\
\P\{\zeta>x\} &=& \min\biggl(1,\frac{\P\{\log(1+A+B)>x\}}{\P\{A>0\}}\biggr),
\quad x>0.
\end{eqnarray*}
Both are subexponential random variables provided $\log(1+A+B)$ is so,
see e.g. \cite[Corollary 3.13]{FKZ}.
It follows from \eqref{bound.xi.x} that, for all $x>0$,
\begin{eqnarray*}
\P\{\xi(x)>y\mid A=0\} &\le& \P\{\eta>y\},\\ 
\P\{\xi(x)>y\mid A>0\} &\le& \P\{\zeta>y\},
\end{eqnarray*}
which implies that
\begin{eqnarray*}
\P\{X_{k+1}>x\mid A_1=0,A_2,\ldots,A_{k+1}>0\} &\le& 
\P\{\eta+\zeta_1+\ldots+\zeta_k>x\},
\end{eqnarray*}
where $\zeta_i$'s are independent copies of $\zeta$ independent of $\eta$.
Then standard technique based on Kesten's bound for convolutions
of subexponential distributions, see e.g. Theorem 3.39 in \cite{FKZ}, 
allows us to deduce from \eqref{Xn.last.zero} that, for any fixed $\gamma>0$,
\begin{eqnarray*}
\lefteqn{\P\{X_n>x\}}\\
&\le& (1+\gamma)\Bigl((1-p_0)^n n \overline G_0(x)
+p_0\sum_{k=0}^{n-1} (1-p_0)^k(\overline G_0(x)+k\overline H_0(x)\Bigr)
\end{eqnarray*}
for all $n\ge 1$ and sufficiently large $x$. Therefore,
\begin{eqnarray*}
\P\{X_n>x\} &\le& (1+\gamma)\frac{1-(1-p_0)^n}{p_0} \overline H(x),
\end{eqnarray*}
which together with the lower bound proves \eqref{thm.upper.2.5}.
\end{proof}

\section{The case of positive $A$ and signed $B$}
\label{sec:subexp.2}

In this section we consider the case where $D_n$ takes both positive
and negative values because of singed $B$, while $A$ is still assumed 
positive in this section, $A>0$. 
The Markov chain $X_n$ is defined as in \eqref{def.X.signed}.

As $B$ is no longer assumed positive, it makes the tail behaviour of $D$ quite 
different if no further assumptions are made on dependency between $A$ and $B$.
For example, in the extreme case where $B=-cA$ for some $c>0$,
so $D_{n+1}=A_n(D_n-c)$, we have that $D_n$ is eventually negative,
$D_\infty<0$ with probability 1.

More generally, if $B=A\eta$ where $\eta$ is independent of $A$ and takes 
values of both signs, then we conclude similar to \eqref{thm.lower.1} that,
as $x\to\infty$,
\begin{eqnarray*}
\P\{D_\infty>x\} &\ge& 
\biggl(\frac{1}{a}\int_\R \P\{\eta>-c\}\P\{D_\infty\in dc\}+o(1)\biggr)
\overline{F_I}(\log x),
\end{eqnarray*}
provided the distribution $F_I$ is long-tailed.
However, the technique used in Section \ref{sec:subexp.1} for proving the matching 
upper bound does not work in such cases as the Lindley majorant returns
the coefficient $a^{-1}$ which is greater than that in the lower bound above.
For that reason we restrict further considerations
to the case where $A$ and $B$ are independent.

\begin{Theorem}\label{thm:subexp.gen.B}
Suppose that $A>0$, $A$ and $B$ are independent, 
$\E\xi=-a\in(-\infty,0)$ and $\E\log(1+|B|)<\infty$.

If the integrated tail distributions $F_I$ and $G^+_I$ are long-tailed, then
\begin{eqnarray}\label{thm:subexp.gen.B.lower.1}
\P\{D_\infty>x\} &\ge& (a^{-1}{+}o(1))\Bigl(
\P\{D_\infty>0\}\overline{F_I}(\log x)
+\overline{G^+_I}(\log x)\Bigr)\mbox{ as }x\to\infty.\nonumber\\[-2mm]
\end{eqnarray}
If, in addition, the distributions $F$and $G^+$ are long-tailed itself, 
then, as $x$, $n\to\infty$,
\begin{eqnarray}\label{thm:subexp.gen.B.lower.2}
\lefteqn{\P\{D_n>x\}}\nonumber\\ 
&\ge& \frac{1{+}o(1)}{a}\biggl(
\P\{D_\infty>0\}\int_{\log x}^{\log x+na}\overline F(y)dy
+\int_{\log x}^{\log x+na}\overline{G^+}(y)dy\biggr).
\end{eqnarray}

If $\P\{D_\infty=0\}=0$, the integrated tail distributions 
$F_I$, $G^+_I$ and $G^-_I$ are long-tailed,
$\overline{G^-_I}(z)=O(\overline{F_I}(z)+\overline{G^+_I}(z))$ 
and $H_I$ is subexponential then
\begin{eqnarray}\label{thm:subexp.gen.B.upper.1}
\P\{D_\infty>x\} &\sim& a^{-1}\Bigl(
\P\{D_\infty>0\}\overline{F_I}(\log x)
+\overline{G^+_I}(\log x)\Bigr)
\mbox{ as }x\to\infty.
\end{eqnarray}
If moreover the distributions $F$, $G^+$ and $G^-$ are long-tailed,
$\overline{G^-}(z)=O(\overline F(z)+\overline{G^+}(z))$
and $H$ is strong subexponential then, as $x$, $n\to\infty$,
\begin{eqnarray}\label{thm:subexp.gen.B.upper.2}
\P\{D_n>x\} &\sim& \frac{1}{a}\biggl(
\P\{D_\infty>0\}\int_{\log x}^{\log x+na}\overline F(y)dy
+\int_{\log x}^{\log x+na}\overline{G^+}(y)dy\biggr).\nonumber\\[-1mm]
\end{eqnarray}
\end{Theorem}

\begin{proof}
Fix an $\varepsilon>0$. As follows from \eqref{def.xi.signed+}, for $x\ge 0$, 
\begin{eqnarray*}
\xi(x) &\ge& 
\left\{
\begin{array}{rl}
\log(A(1-e^{-x})-e^{-x}B^-) & \mbox{ if }A(e^x-1)+B\ge 0,\\
-\log(1+A+|B|) & \mbox{ if }A(e^x-1)+B<0,
\end{array}
\right.
\end{eqnarray*}
where the second line follows due to $A>0$.
The minorant on the right hand side is stochastically increasing as $x$ grows,
therefore, there exists a sufficiently large $x_0$ and a random variable 
$\eta$ such that
\begin{eqnarray}\label{xi.x.min.B}
\xi(x) &\ge& \eta\quad\mbox{for all }x\ge x_0
\ \mbox{ and }\ \E\eta>-a-\varepsilon/2.
\end{eqnarray}

As in the last proof, we start with the lower bound 
\eqref{thm:subexp.gen.B.lower.1} following the single big jump technique.
Since $D_n$ is assumed to be convergent,
the associated Markov chain $X_n$ is stable, 
so there exist $n_0$ and $c>2$ such that
\begin{eqnarray*}
\P\{X_n\in(1/c,c]\} &\ge& (1-\varepsilon) \P\{D_\infty>0\}
\quad\mbox{for all }n\ge n_0,\\
\P\{|X_n|\le c\} &\ge& 1-\varepsilon\quad\mbox{for all }n,
\end{eqnarray*}
and also $\P\{A\le c\}\ge 1-\varepsilon$, $\P\{|B|\le c\}\ge 1-\varepsilon$.
For all $k$, $n$ and $c$, let us consider the events $\Omega(k,n,c)$ defined in 
\eqref{def.Bknc} and satisfying \eqref{B.c}.
It follows from \eqref{xi.x.min.B} that any of the events
\eqref{events.lower.bound} implies $X_n>x$ and they are pairwise disjoint.
Therefore, by the Markov property and \eqref{B.c}, 
\begin{eqnarray}\label{B.arb.I12}
\lefteqn{\P\{X_n>x\}}\nonumber\\ 
&\ge& \sum_{k=1}^n \P\{X_{k-1}\le c,\ 
X_k>x+c+(n-k)(a+\varepsilon)\}\P\{\Omega(k,n-k,c)\}\nonumber\\
&\ge& (1-\varepsilon)\sum_{k=1}^n \P\{X_{k-1}\le c,\ 
X_k>x+c+(n-k)(a+\varepsilon)\},
\end{eqnarray}
The $k$th term of the sum is not less than
\begin{eqnarray*}
\lefteqn{\biggl(\int_{-c}^0+\int_0^c\biggr)
\P\{X_{k-1}\in dy\}\P\{y+\xi(y)>z_{n-k}\}}\\
&=&  \int_{-c}^0
\P\{X_{k-1}\in dy\}\P\{\log(1+A(1-e^{-y})+B)>z_{n-k}\}\\
&& + \int_0^c \P\{X_{k-1}\in dy\}\P\{\log(1+A(e^y-1)+B)>z_{n-k}\}\\
&=:& I_1+I_2,
\end{eqnarray*}
where $z_k=x+c+k(a+\varepsilon)$.
For all $y\in[-c,0]$ and $z>0$, 
owing to the condition $A>0$ and independence of $A$ and $B$
\begin{eqnarray*}
\P\{\log(1+A(1-e^{-y})+B)>z\} &\ge& \P\{\log(1-Ae^c+B)>z\}\\
&\ge& \P\{A\le c\}\P\{\log(1-ce^c+B)>z\}\\
&\ge& \P\{A\le c\}\overline{G^+}(z+1)
\end{eqnarray*}
for all sufficiently large $z$ which yields that
\begin{eqnarray}\label{B.arb.I1}
I_1 &\ge& \P\{A\le c\}\P\{X_{k-1}\in[-c,0]\}\overline{G^+}(z_{n-k}+1)\nonumber\\
&\ge& (1-\varepsilon)\P\{X_{k-1}\in[-c,0]\}\overline{G^+}(z_{n-k}+1),
\end{eqnarray}
due to the choice of $c$. For all $y>0$, 
\begin{eqnarray*}
\lefteqn{\P\{\log(1+A(e^y-1)+B)>z\}}\\
&\ge& \P\{|B|\le c\}\P\{\log(1+A(e^y-1)-c)>z\} +\P\{\log(1+B)>z\},
\end{eqnarray*}
which yields that
\begin{eqnarray*}
I_2 &\ge& 
\P\{|B|\le c\}\int_{1/c}^c \P\{\log(1+A(e^y-1)-c)>z_{n-k}\}\P\{X_{k-1}\in dy\}\\
&&\hspace{60mm}+\overline{G^+}(z_{n-k})\P\{X_{k-1}\in(0,c]\}\\
&\ge& (1-\varepsilon)\P\{\log(1+A(e^{1/c}-1)-c)>z_{n-k}\}\P\{X_{k-1}\in(1/c,c]\}\\
&&\hspace{60mm}+\overline{G^+}(z_{n-k})\P\{X_{k-1}\in(0,c]\}.
\end{eqnarray*}
Therefore, by the choice of $c$, for all sufficiently large $x$ and $k>n_0$,
\begin{eqnarray}\label{B.arb.I2}
I_2 &\ge& (1-\varepsilon)^2 \P\{D_\infty>0\}\overline F(z_{n-k}+1)
+\overline{G^+}(z_{n-k})\P\{X_{k-1}\in(0,c]\}.\nonumber\\[-1mm]
\end{eqnarray}
Substituting \eqref{B.arb.I1} and \eqref{B.arb.I2} 
into \eqref{B.arb.I12} we deduce that
\begin{eqnarray*}
\P\{X_n>x\} &\ge& (1-\varepsilon)^2
\sum_{k=n_0+1}^n \Bigl(\P\{D_\infty>0\}
\overline F(x+c+1+(n-k)(a+\varepsilon))\\
&&\hspace{40mm} +\overline{G^+}(x+c+1+(n-k)(a+\varepsilon))\Bigl)
\end{eqnarray*}
Since the tail is a non-increasing function, the last sum is not less than
\begin{eqnarray}\label{integral.0.n.B}
\frac{1}{a+\varepsilon}
\int_0^{(n-n_0-1)(a+\varepsilon)} \Bigl(\P\{D_\infty>0\}
\overline F(x+c+1+y)+\overline{G^+}(x+c+1+y)\Bigr)dy.\nonumber\\[-2mm]
\end{eqnarray}
Letting $n\to\infty$ we obtain that the tail at point $x$ 
of the stationary distribution of the Markov chain $X$ is not less than
\begin{eqnarray}\label{lower.pre.c}
\lefteqn{\frac{(1-\varepsilon)^2}{a+\varepsilon}
\int_0^\infty \Bigl(\P\{D_\infty>0\}
\overline F(x+c+1+y)+\overline{G^+}(x+c+1+y)\Bigr)dy}\nonumber\\
&=& \frac{(1-\varepsilon)^2}{a+\varepsilon} \Bigl(\P\{D_\infty>0\}
\overline{F_I}(x+c+1)+\overline{G_I^+}(x+c+1)\Bigr)\\
&\sim& \frac{(1-\varepsilon)^2}{a+\varepsilon} 
\Bigl(\P\{D_\infty>0\}\overline{F_I}(x)+\overline{G_I^+}(x)\Bigr)
\quad\mbox{as }x\to\infty,\nonumber
\end{eqnarray}
due to the long-tailedness of the integrated tail distributions $F_I$
and $G_I^+$.
Summarising altogether we deduce that, for every fixed $\varepsilon>0$,
\begin{eqnarray*}
\liminf_{x\to\infty}
\frac{\P\{D_\infty>x\}}{\P\{D_\infty>0\}
\overline{F_I}(\log x)+\overline{G_I^+}(\log x)}
&\ge& \frac{(1-\varepsilon)^2}{a+\varepsilon},
\end{eqnarray*}
which implies the lower bound \eqref{thm:subexp.gen.B.lower.1}
due to the arbitrary choice of $\varepsilon>0$.

If the distributions $F$ and $G^+$ are long-tailed itself, then the integral in 
\eqref{integral.0.n.B} is asymptotically equivalent to the integral
\begin{eqnarray*}
\int_x^{x+n(a+\varepsilon)} \Bigl(\P\{D_\infty>0\}
\overline F(y)+\overline{G^+}(y)\Bigr)dy
\quad\mbox{as }x,\ n\to\infty,
\end{eqnarray*}
and the second lower bound \eqref{thm:subexp.gen.B.lower.2} follows too.

To prove matching upper bounds let us first observe that
\begin{eqnarray}\label{mod.D.upper}
|D_{n+1}| &\le& A_n |D_n|+|B_n|\quad\mbox{for all }n,
\end{eqnarray}
where the right hand side is increasing in $D_n$. Hence,
$|D_n|\le \widetilde D_n$, where $\widetilde D_n$ 
is a positive stochastic difference recursion,
\begin{eqnarray*}
\widetilde D_{n+1} &=& A_n \widetilde D_n+|B_n|.
\end{eqnarray*}
Since $H_I$ is subexponential, 
Theorem \ref{thm:subexp} applies to $\widetilde D_n$, so
\begin{eqnarray*}
\P\{\widetilde D_\infty>x\} &\sim& 
a^{-1}\overline{H_I}(\log x)\quad\mbox{as }x\to\infty,
\end{eqnarray*}
and hence
\begin{eqnarray*}
\P\{|D_\infty|>x\} &\le& 
(a^{-1}+o(1))\overline{H_I}(\log x)\quad\mbox{as }x\to\infty,
\end{eqnarray*}
It follows from \eqref{H.above.gen} that
\begin{eqnarray*}
\overline H(x) &\le& \P\{\log(1+A)>x-1\}+\P\{\log(1+|B|)>x-1\}.
\end{eqnarray*}
Integrating the last inequality we get an upper bound
\begin{eqnarray}\label{H.A.B.pre}
\overline{H_I}(x) &\le& 
\overline{F_I}(x-1)+\overline{G^-_I}(x-1)+\overline{G^+_I}(x-1)\\
&\sim& \overline{F_I}(x)+\overline{G^-_I}(x)+\overline{G^+_I}(x)
\quad\mbox{as }x\to\infty,\nonumber
\end{eqnarray}
because all three distributions, $F_I$, $G^-_I$ and $G^+_I$ are assumed long-tailed. 
Hence the following upper bound holds for the tail of $D_\infty$, as $x\to\infty$:
\begin{eqnarray}\label{mod.D.infty.upper}
\P\{|D_\infty|>x\} &\le& 
(a^{-1}{+}o(1))\bigl(\overline{F_I}(\log x)+\overline{G^-_I}(\log x)
+\overline{G^+_I}(\log x)\bigr).
\end{eqnarray}

The long-tailedness of $F_I$ and $G^-_I$ similarly to 
\eqref{thm:subexp.gen.B.lower.1} implies that
\begin{eqnarray*}
\P\{D_\infty<-x\} &\ge& (a^{-1}+o(1))\bigl(
\P\{D_\infty<0\}\overline{F_I}(\log x)
+\overline{G^-_I}(\log x)\bigr),
\end{eqnarray*}
and the two lower bounds together imply that, as $x\to\infty$,
\begin{eqnarray*}
\P\{|D_\infty|>x\} &\ge& (a^{-1}+o(1))\bigl(
\overline{F_I}(\log x)+\overline{G^+_I}(\log x)
+\overline{G^-_I}(\log x)\bigr),
\end{eqnarray*}
because $\P\{D_\infty=0\}=0$. Together with the upper bound
\eqref{mod.D.infty.upper} it yields that
\begin{eqnarray*}
\P\{D_\infty>x\} &=& a^{-1}\bigl(\P\{D_\infty>0\}
\overline{F_I}(\log x)+\overline{G^+_I}(\log x)\bigr)
+o(\overline{H_I}(\log x)),
\end{eqnarray*}
and the first asymptotics \eqref{thm:subexp.gen.B.upper.1} follows
by the condition 
$\overline{G^-_I}(z)=O(\overline{F_I}(z)+\overline{G^+_I}(z))$.

The second asymptotics \eqref{thm:subexp.gen.B.upper.2} follows 
along similar arguments.
\end{proof}

\section{Balance of negative and positive tails in the case of signed $A$}
\label{sec:subexp.3}

In this section we turn to the general case where $D_n$ takes both positive
and negative values, with $A$ taking values of both signs.
Denote $\xi:=\log |A|$ and the distribution of $\log(1+|A|)$ by $F$.
Recall that the distribution of $\log(1+|B|)$ is denoted by $G$
and the distribution of $\log(1+|A|+|B|)$ by $H$.

The Markov chain $X_n$ is defined as above in \eqref{def.X.signed}.

\begin{Theorem}\label{thm:subexp.gen}
Suppose that $\P\{D_\infty=0\}=0$,
\begin{eqnarray}\label{thm:positivity.p-.p+}
0\ <\ \P\{A>0\}\ <\ 1,
\end{eqnarray}
$A$ and $B$ are independent, $\E\xi=-a\in(-\infty,0)$ and $\E\log(1+|B|)<\infty$.

If the integrated tail distribution $H_I$ is long-tailed, then
\begin{eqnarray}\label{thm:subexp.gen.lower.1}
\P\{D_\infty>x\} &\ge& (1/2a+o(1))\overline{H_I}(\log x)
\quad\mbox{as }x\to\infty.
\end{eqnarray}
If, in addition, the distribution $H$ is long-tailed itself, then
\begin{eqnarray}\label{thm:subexp.gen.lower.2}
\P\{D_n>x\} &\ge& \frac{1+o(1)}{2a}
\int_{\log x}^{\log x+na}\overline H(y)dy\quad\mbox{as }n,\ x\to\infty.
\end{eqnarray}

If the integrated tail distribution $H_I$ is subexponential then
\begin{eqnarray}\label{thm:subexp.gen.upper.1}
\P\{D_\infty>x\} &\sim& \frac{1}{2a}\overline{H_I}(\log x)
\quad\mbox{as }x\to\infty.
\end{eqnarray}
If moreover the distribution $H$ is strong subexponential then
\begin{eqnarray}\label{thm:subexp.gen.upper.2}
\P\{D_n>x\} &\sim& \frac{1}{2a}\int_{\log x}^{\log x+na}\overline H(y)dy
\quad\mbox{as }n,\ x\to\infty.
\end{eqnarray}
\end{Theorem}

\begin{proof}
The same arguments based on the single big jump technique used
in the last section for proving \eqref{lower.pre.c} show that, 
for any fixed $\varepsilon>0$, there exists a $c<\infty$ such that
\begin{eqnarray*}
\P\{|X_\infty|>x\} &\ge& \frac{1-\varepsilon}{a}
\bigl(\P\{D_\infty\not=0\}\overline{F_I}(x+c+1)+\overline{G_I}(x+c+1)\bigr)
\end{eqnarray*}
for all sufficiently large $x$. Similar to \eqref{H.A.B.pre},
\begin{eqnarray*}
\overline{H_I}(x) &\le& 
\overline{F_I}(x-1)+\overline{G_I}(x-1)
\end{eqnarray*}
for all sufficiently large $x$,
which together with the condition $\P\{D_\infty=0\}=0$ implies that
\begin{eqnarray*}
\P\{|X_\infty|>x\} &\ge& \frac{1-\varepsilon}{a}\overline{H_I}(x+c+2)\\
&\sim& \frac{1-\varepsilon}{a}\overline{H_I}(x)\quad\mbox{as }x\to\infty,
\end{eqnarray*}
due to the long-taileness of the distribution $H_I$. Therefore,
\begin{eqnarray}\label{mod.D.lower}
\P\{|X_\infty|>x\} &\ge& (a^{-1}+o(1))\overline{H_I}(x)\quad\mbox{as }x\to\infty.
\end{eqnarray}
At any time large absolute value of $X_n$ changes its sign with 
asymptotic (as $x\to\infty$) probability $p^-=\P\{A<0\}$ 
and keeps its sign with asymptotic probability $p^+=\P\{A>0\}$,
so sign change may be asymptotically described as
a Markov chain with transition probability matrix
\begin{eqnarray*}
\left(
\begin{array}{cc}
p^+ & p^-\\
p^- & p^+
\end{array}
\right),
\end{eqnarray*}
whose asymptotic distribution is $(1/2,1/2)$, owing to the condition 
\eqref{thm:positivity.p-.p+}. 
For that reason, the probability of a large positive value of $X_n$
is approximately at least one half of the right hand side of \eqref{mod.D.lower},
and the proof of \eqref{thm:subexp.gen.lower.1} is complete.
The proof of \eqref{thm:subexp.gen.lower.2} follows the same lines.

To prove the upper bound \eqref{thm:subexp.gen.upper.1},
similar to \eqref{mod.D.upper} we first note that
\begin{eqnarray*}
|D_{n+1}| &\le& |A_n| |D_n|+|B_n|\quad\mbox{for all }n,
\end{eqnarray*}
which allows to conclude the proof as it was done in the last section.
\end{proof}

\end{document}